\documentclass[a4paper,11pt]{amsart}
\usepackage{amssymb}
\usepackage{mathrsfs}
\usepackage{amsmath,amssymb,amsthm,latexsym,amscd,mathrsfs}
\usepackage{indentfirst}
 \newcommand{\ROM}[1]{\mathrm{\uppercase\expandafter{\romannumeral#1}}}
\numberwithin{equation}{section} \theoremstyle{plain}
\newtheorem{thm}{Theorem}[section]

\newtheorem{pro}{Proposition}[section]
\newtheorem{lem}{Lemma}[section]
\newtheorem{cor}{Corollary}[section]
\theoremstyle{definition}

\theoremstyle{remark}
\newtheorem{rem}{Remark}[section]

\newtheorem{ack}{Acknowledgements}   

\title[Anisotropic isoparametric hypersurfaces]
{Anisotropic isoparametric hypersurfaces in Euclidean spaces}
\author[J.Q. Ge]{Jianquan Ge}
\address{School of Mathematical Sciences, Laboratory of Mathematics and Complex Systems, Beijing Normal
University, Beijing 100875, P.R. CHINA} \email{jqge@bnu.edu.cn}
\author[H. Ma]{Hui Ma}
\address{Department of Mathematical Sciences, Tsinghua University,
Beijing 100084, P.R. CHINA} \email{hma@math.tsinghua.edu.cn}
\date{}
\thanks{The first author is partially supported by NSFC grant No.~11001016 and Innovative Team Program of Ministry of Education of China.
The second author is partially supported by NSFC grant No.~10501028
and NKBRPC No.~2006CB805905.}
\begin{document}
\maketitle

\begin{abstract}
In this note, we give a classification of complete anisotropic
isoparametric hypersurfaces, \emph{i.e.}, hypersurfaces with
constant anisotropic principal curvatures, in Euclidean spaces,
which is in analogue with the classical case for isoparametric
hypersurfaces in Euclidean spaces. On the other hand, by an example
of local anisotropic isoparametric surface constructed by B. Palmer,
we find that anisotropic isoparametric hypersurfaces have both local
and global aspects as in the theory of proper Dupin hypersurfaces.
\end{abstract}

\section{Introduction}
\label{introduction}

Let $F: \mathbb{S}^n \rightarrow \mathbb{R}^{+}$ be a smooth
positive function defined on the unit sphere satisfying the
following convexity condition:
\begin{equation}\label{covexity}
A_F:=(D^2 F+F I)_u >0,
\end{equation}
for any $u\in \mathbb{S}^n$, where $D^2 F$ denotes the Hessian of
$F$ on $\mathbb{S}^n$, $I$ denotes the identity on $T_u
\mathbb{S}^n$ and $>0$ means the matrix is positive definite. Let
$x: M\rightarrow \mathbb{R}^{n+1}$ be an immersed oriented
hypersurface without boundary and $\nu: M\rightarrow \mathbb{S}^n$
denote its Gauss map. Then \emph{anisotropic surface energy} (of
$x$) is a parametric elliptic functional $\mathcal{F}$ defined as
follows:
\begin{equation*}
\mathcal{F}(x)=\int_M F(\nu)dA.
\end{equation*}
Note that if $F\equiv 1$, then $\mathcal{F}(x)$ is just the area of
$x$.

The critical points of $\mathcal{F}$ for all compactly supported
volume-preserving variations are characterized by the property that
the \emph{anisotropic mean curvature} $H_F$ is constant, where $H_F$
is given by
\begin{equation*}
n H_F=nHF-{\rm div}_M DF=-{\rm tr}_M d(\phi\circ \nu).
\end{equation*}
Here $\phi$ is defined in (\ref{phi}) below and
\begin{equation}\label{SF}
S_F:=-d(\phi\circ \nu)=-A_F\circ d\nu=A_F\circ T\circ dx
\end{equation}
is called the $F$-\emph{Weingarten operator}\footnote{For
simplicity, sometimes we write the operators without $dx$ or
identify them with corresponding matrix representations.} with
respect to the induced metric of $x$, where $T=-d\nu$ is the
Weingarten (shape) operator of $M$. In general $S_F$ is not
self-dual, but it still has real eigenvalues $\lambda_1, \cdots,
\lambda_n$, which are called \emph{anisotropic principal
curvatures}. If $\lambda_1=\lambda_2=\cdots=\lambda_n$ holds
everywhere on $M$, $M$ is called \emph{totally anisotropic
umbilical} for $F$ or for $\mathcal{F}$. Similarly, $M$ is called
\emph{anisotropic isoparametric} for $F$ or for $\mathcal{F}$, if
its anisotropic principal curvatures are constant.

A fundamental result relating to the anisotropic surface energy is
the Wulff's theorem, which states that among all closed
hypersurfaces enclosing the same volume, there exists an absolute
minimizer $W_F$ of $\mathcal{F}$ (cf. \cite{BM}, \cite{Taylor}).
Here $W_F$ is the so-called ``Wulff shape" that can be defined as
follows. Consider the map
\begin{equation}\label{phi}
\begin{split}
\phi: \mathbb{S}^n & \rightarrow \mathbb{R}^{n+1}\\
  u~& \mapsto DF_u+F(u)u,
\end{split}
\end{equation}
where $DF$ is the gradient of $F$ on $\mathbb{S}^n$. Then
$W_F:=\phi(\mathbb{S}^{n})$ is called the \emph{Wulff shape} of $F$
or $\mathcal{F}$ (cf. \cite{Gi}, \cite{Cv}, \cite{Winklmann}, etc.).
Under the convexity condition of $F$, $W_F$ is a smooth convex
hypersurface. When $F\equiv 1$, the Wulff shape is just the unit
sphere $\mathbb{S}^n$. An equivalent definition of the Wulff shape
can be given in terms of the dual norm $F^{*}$ of $F$, where $F^{*}:
\mathbb{R}^{n+1}\rightarrow \mathbb{R}$ is defined by (\cite{HLMG})
\begin{equation*}
F^{*}(y)=\sup\{\frac{\langle y,z\rangle}{F(z)}{\ } \vert {\ } z\in
\mathbb{S}^n\}.
\end{equation*}
Thus the Wulff shape $W_F$ is just the unit sphere under this norm
$F^{*}$, \emph{i.e.},
\begin{equation*}
W_F=\{y\in \mathbb{R}^{n+1}{\ }\vert{\ } F^{*}(y)=1\}.
\end{equation*}
The Wulff shape plays the same role as the unit sphere for the area
functional. For example, up to translations and homotheties, $W_F$
is the only closed stable (for the functional $\mathcal{F}$)
hypersurface immersed in $\mathbb{R}^{n+1}$ with constant
anisotropic mean curvature (\cite{Palmer},\cite{K-Pa1}), and also
the only closed embedded hypersurface with constant anisotropic mean
curvature (\cite{HLMG},\cite{M}). Moreover, in the case of dimension
$2$, $W_F$ is the only topological $2$-sphere immersed in
$\mathbb{R}^{3}$ with constant anisotropic mean curvature
(\cite{K-Pa2}).

As to the classification of totally anisotropic umbilical
hypersurfaces, we have the following
\begin{pro}(cf. \cite{HLMG}, \cite{Palmer})\label{totally umbilical}
Any totally anisotropic umbilical hypersurface immersed in
$\mathbb{R}^{n+1}$ $(n\geq2)$, or $\lambda_1\equiv constant\neq0$
$(n=1)$, is an open part of a hyperplane or the Wulff shape, up to
translations and homotheties.
\end{pro}
Comparing with the classical hypersurface theory,
 it is natural to ask the classification problem of anisotropic
isoparametric hypersurfaces in Euclidean spaces. In this note, we
will give such a classification which was conjectured to be parallel
with the classical case of isoparametric hypersurfaces in Euclidean
spaces.

For each totally geodesic $k$-dimensional sphere
$\mathbb{S}^{k}\subset\mathbb{S}^{n}$, we denote by
$W^k_F:=\phi(\mathbb{S}^{k})\subset W_F$ being its image under the
map $\phi$ defined in (\ref{phi}). It is easily seen that $W^k_F$ is
a $k$-dimensional submanifold of the Wulff shape $W_F$ and depends
on the choice of the inclusion
$\mathbb{S}^{k}\subset\mathbb{S}^{n}$. Note that in general $W^k_F$
may not equal to either the Wulff shape $\overline{W^k_F}$ of
$F|_{\mathbb{S}^{k}}$ in $\mathbb{R}^{k+1}\subset\mathbb{R}^{n+1}$,
or the intersection of $W_F$ with a $(k+1)$-dimensional hyperplane.
However, there is a family of natural embeddings of
$W^k_F\times\mathbb{R}^{n-k}$ in $\mathbb{R}^{n+1}$ as follows. For
any $t\neq0$, define
\begin{equation}\label{subwulff-imm}
\begin{split}
\varphi_t: W^k_F\times\mathbb{R}^{n-k} & \rightarrow \mathbb{R}^{n+1}\\
  (\phi(u)~,~v)\quad& \mapsto v-t\phi(u),
\end{split}
\end{equation}
where $u\in\mathbb{S}^{k}\subset\mathbb{R}^{k+1}$ and
$\mathbb{R}^{n-k}\subset\mathbb{R}^{n+1}$ is the orthogonal subspace
of the given $\mathbb{R}^{k+1}$. It is not hard to check that
$\varphi_t$ is an immersion and has two distinct constant
anisotropic principal curvatures $1/t,0$ with multiplicities $k,n-k$
respectively. In fact, $\varphi_t$ is an embedding and its image
coincides with that of the natural inclusion
$\overline{W^k_F}\times\mathbb{R}^{n-k}\subset\mathbb{R}^{n+1}$ up
to translations and homotheties, since the image of the projection
of $W^k_F$ to $\mathbb{R}^{k+1}$ is just the Wulff shape
$\overline{W^k_F}$.
 Then
our main result can be stated as follows.
\begin{thm}\label{thm}
A complete hypersurface in the Euclidean space $\mathbb{R}^{n+1}$
has constant anisotropic principal curvatures if and only if up to
translations and homotheties, it is
\begin{enumerate}
\item $\mathbb{R}^{n}\subset\mathbb{R}^{n+1}$, or
\item $W_F\subset\mathbb{R}^{n+1}$, or
\item $W^k_F\times\mathbb{R}^{n-k}\xrightarrow[]{\varphi_t}\mathbb{R}^{n+1}$ for some $0<k<n$, $t\neq0$, and $\mathbb{S}^{k}\subset\mathbb{S}^{n}$.
\end{enumerate}
\end{thm}
Note that when $F\equiv1$, this theorem reduces to the well-known
classification of isoparametric hypersurfaces in Euclidean spaces.
For the history and recent progresses of isoparametric hypersurfaces
in real space forms, we would like to refer to the excellent surveys
given by Thorbergsson \cite{Th00} and Cecil \cite{Ce}. From them one
can also find introductions of many important generalizations of
isoparametric hypersurfaces such as (proper) Dupin hypersurfaces. In
contrary with isoparametric hypersurfaces, proper Dupin
hypersurfaces have both local and global aspects, such as the number
$g$ of distinct principal curvatures could have different sets of
values for local and compact proper Dupin hypersurfaces, which
happens to appear also in our situation. In fact, in a letter to us
B. Palmer \cite{Palmer2} showed that there is a rotationally
symmetric anisotropic surface energy functional $\mathcal {F}$ (or
corresponding $F$) such that a part $U$ of the helicoid $\Sigma$ in
$\mathbb{R}^3$ has constant anisotropic principal curvatures $\pm1$.
Then we can extend the corresponding $F$ to a positive function
$\tilde{F}$ on $\mathbb{S}^3$ such that locally it equals $F$ along
the normal geodesics starting from the Gauss image $\mathcal {U}$
($\subset \mathbb{S}^2\subset \mathbb{S}^3$) of $U$, \emph{i.e.},
$$\tilde{F}(\cos tu+\sin te_4):=F(u),\quad \emph{for any}~
u\in\mathcal {U},~ t\in(-\varepsilon,\varepsilon),$$ where $e_4$ is
the fourth coordinate vector of
$\mathbb{R}^4=\mathbb{R}^3\oplus\mathbb{R}$. It is easily verified
that the function $\tilde{F}$ satisfies the convexity condition
(\ref{covexity}) on a small neighborhood of $\mathcal {U}$ in
$\mathbb{S}^3$ (and thus could be extended to the whole
$\mathbb{S}^3$ satisfying (\ref{covexity})). For this function
$\tilde{F}$, one can see that (a part of) the canonical embedding
$U\times\mathbb{R}\hookrightarrow\mathbb{R}^4$ has three constant
anisotropic principal curvatures $1$, $-1$ and $0$. Hence, one can
obtain local anisotropic isoparametric hypersurfaces in
$\mathbb{R}^{n+1}$ with more than $2$ distinct anisotropic principal
curvatures, which differs with the complete case as classified in
Theorem \ref{thm}.

\section{Anisotropic parallel
translation}\label{aniso-parallel-trans}

In this section, we will study the geometry of ``anisotropic
parallel translation" of hypersurfaces in $\mathbb{R}^{n+1}$ and
apply it to establish some preliminary results. Throughout of this
paper, we will use the same notations as those in Section
\ref{introduction}.

Let $x_t: M\rightarrow \mathbb{R}^{n+1}$ be a family of small
perturbations of $x$ ($x_0=x$) defined by
\begin{equation}\label{parallel}
x_t=x+t\phi\circ \nu.
\end{equation}
Comparing with standard parallel translation, we call such
perturbation $x_t$ an \emph{anisotropic parallel translation}.
Taking differential of (\ref{parallel}), we get
\begin{equation}\label{d parallel}
dx_t=(I-tS_F)\circ dx.
\end{equation}
Thus, for any $t$ small enough, $x_t$ remains an immersion and $\nu$
(up to parallel translations) is still the unit normal vector field
of $x_t$. Then by (\ref{SF}) the $F$-Weingarten operator for $x_t$,
say $S^t_F$, can be expressed as
\begin{equation}\label{F-shape for xt}
S^t_F=S_F\circ dx=S_F\circ(I-tS_F)^{-1}\circ dx_t,
\end{equation}
where $(I-tS_F)^{-1}$ is the inverse of the operator $(I-tS_F)$.
If $S_F X=\lambda X$, this becomes
$S^t_F X=\frac{\lambda}{1-t\lambda} X$.
Therefore, the anisotropic principal curvatures of $x_t$, say
$\lambda_1(t),\cdots,\lambda_n(t)$, can be expressed
in terms of those of $x$ as follows:
\begin{equation}\label{Principal curv of xt}
\lambda_i(t)=\frac{\lambda_i}{1-t\lambda_i},\quad i=1,\cdots,n.
\end{equation}
By standard method as Cartan and Nomizu did (cf. for example,
\cite{N},\cite{CR}), from (\ref{Principal curv of xt}) we can
similarly obtain the following result.
\begin{thm}
Let $x_t: M\rightarrow \mathbb{R}^{n+1}$,
$-\varepsilon<t<\varepsilon$, be a family of anisotropic parallel
hypersurfaces in $\mathbb{R}^{n+1}$. Then $x(M)$ has constant
anisotropic principal curvatures if and only if each $x_t(M)$ has
constant anisotropic mean curvature.
\end{thm}
\begin{proof}
For the sake of completeness, we give a simple proof as follows.

Let $H_F(t)$ be the anisotropic mean curvature of $x_t(M)$. Then by
(\ref{Principal curv of xt}), we have
\begin{eqnarray}\label{aniso-mean-form}
nH_F(t)&=&tr_{x_t(M)}S^t_F=\sum_{i=1}^n\lambda_i(t)=\sum_{i=1}^n\frac{\lambda_i}{1-t\lambda_i}\nonumber\\
&=&\sum_{i=1}^n(-\log(1-t\lambda_i))'=-\left(\log\left(\prod_{i=1}^n(1-t\lambda_i)
\right)\right)'\nonumber\\
&=&-\left(\log\left(\sum_{k=0}^n(-1)^kM_kt^k \right)\right)',
\end{eqnarray}
where $'$ denotes taking derivative with respect to $t$, and
$M_0\equiv1$, $M_1,\cdots,M_n$ are the elementary symmetric
polynomials of $\lambda_1,\cdots,\lambda_n$. Then we can get our
conclusions from (\ref{aniso-mean-form}) directly, since it implies
that $H_F(t)\equiv const$ on $x_t(M)$ for each
$-\varepsilon<t<\varepsilon$ is equivalent to $M_k\equiv const$ on
$x(M)$ for $k=1,\cdots,n$ and hence is equivalent to
$\lambda_i\equiv const$ on $x(M)$ for $i=1,\cdots,n$.
\end{proof}
Now suppose $x(M)$ has distinct constant anisotropic principal
curvatures $\lambda_1,\cdots,\lambda_g$ with multiplicities
$m_1,\cdots,m_g$, \emph{i.e.}, it is anisotropic isoparametric. Let
$$D_i(p):=\{X\in \mathcal {T}_pM|S_FX=\lambda_iX\}, \quad p\in M.$$
Then we obtain distributions $D_1,\cdots,D_g$ of dimensions
$m_1,\cdots,m_g$ on $M$. Similarly as in standard isoparametric
hypersurface theory, we establish the following.
\begin{pro}\label{prop}
 Each distribution $D_i$ is
integrable and when $\lambda_i\neq0$, the maximal integral manifold
$L_i(p)$ of $D_i$ through $p\in M$ is, up to translations and
homotheties, just a $W^{m_i}_F$ for some $\mathbb{S}^{m_i}\subset
\mathbb{S}^{n}$, or just an open part of it, if $M$ is not complete.
\end{pro}
\begin{proof}
Recall that for a hypersurface $M$ in a real space form the Codazzi
equation takes the form:
$$(\nabla_XT)Y=(\nabla_YT)X, \quad \emph{i.e.},$$
\begin{equation}\label{Codazzi}
\nabla_X(TY)-T(\nabla_XY)=\nabla_Y(TX)-T(\nabla_YX),
\end{equation}
for $X,Y$ tangent to $M$, where $T$ is the shape operator of $M$ and
$\nabla$ is the Levi-Civita connection. Due to the fact that $D_i$
is trivially integrable if $m_i=1$, we assume $m_i\geq2$ in the
following. Let $X,Y$ be linearly independent local vector fields in
$D_i$. Then (\ref{Codazzi}) becomes
$$\nabla_X\lambda_i(A_F^{-1}Y)-T(\nabla_XY)=\nabla_Y\lambda_i(A_F^{-1}X)-T(\nabla_YX),$$
which becomes
\begin{equation}\label{form1}
\lambda_i(\nabla_X(A_F^{-1}Y)-\nabla_Y(A_F^{-1}X))=T[X,Y],
\end{equation}
since $\nabla$ has no torsion and the Lie bracket is
$[X,Y]=\nabla_XY-\nabla_YX$. On the other hand, one can find that
$A_F^{-1}$ is just the shape operator of the Wulff shape $W_F$ in
$\mathbb{R}^{n+1}$ and thus by the Codazzi equation, we have
\begin{equation}\label{Codazzi WF}
\nabla_X(A_F^{-1}Y)-A_F^{-1}(\nabla_XY)=\nabla_Y(A_F^{-1}X)-A_F^{-1}(\nabla_YX),
\end{equation}
where the tangent vectors of $M$ are identified with those of
$\mathbb{S}^n$ under the Gauss map and then with those of $W_F$
under $\phi$. Then combining (\ref{form1}) and (\ref{Codazzi WF}),
we can get
$$(\lambda_iI-S_F)[X,Y]=0.$$
Hence $[X,Y]$ is in $D_i$ which proves that $D_i$ is integrable.

Now for $\lambda_i\neq0$, we consider the anisotropic parallel
translation $x_t$ of $x$ with $t=1/\lambda_i$. From (\ref{d
parallel}) we find that $dx_t$ vanishes on $D_i$ and has constant
rank $n-m_i$. Therefore, $x_t(M)$ is an $(n-m_i)$-dimensional
submanifold in $\mathbb{R}^{n+1}$ and $x_t$ maps each leaf $L_i(p)$
of $D_i$ to one point, say $q\in x_t(M)$, \emph{i.e.},
$$x+t\phi\circ \nu\equiv q, \quad \emph{for }~~ x|_{L_i(p)},$$
or equivalently,
\begin{equation}\label{form WkF}
\lambda_i(x-q)\equiv \phi\circ \nu,\quad \emph{for }~~ x|_{L_i(p)}.
\end{equation}
Notice that when limited on $L_i(p)$, $\nu$ is also a unit normal
vector of $x_t(M)$ at $q$ under parallel translations and thus
defines a map
\begin{equation}\label{v map}
\nu: L_i(p)\rightarrow \mathbb{S}^{m_i}\subset\mathcal
{T}_q^{\perp}(x_t(M)).
\end{equation}
 Finally, the following lemma together with formula
(\ref{form WkF}) will complete the proof of the second part of the
proposition.
\end{proof}
\begin{lem}\label{lemma}
The map $\nu: L_i(p)\rightarrow \mathbb{S}^{m_i}$ defined in (\ref{v
map}) when $\lambda_i\neq0$ is an open map. In particular, if $M$ is
complete, it is a diffeomorphism.
\end{lem}
\begin{proof}
It suffices to prove that the map $\nu: L_i(p)\rightarrow
\mathbb{S}^{m_i}$ is always nondegenerate, \emph{i.e.},
\begin{equation}\label{dv}
d\nu(X)\neq0, \quad \emph{for any} \quad 0\neq X\in D_i.
\end{equation}
 On the other hand, we have
$$d\nu(X)=-T(X)=-A_F^{-1}S_F(X), \quad \emph{for }~~ X\in\mathcal {T}M,$$
which immediately verifies (\ref{dv}) to be correct.
\end{proof}

\section{Proof of the main result}\label{proof}
In this section, we will prove Theorem \ref{thm} based on a
Cartan-type identity which forces the number $g$ of distinct
constant anisotropic principal curvatures must be less than $3$.

Let $x: M\rightarrow \mathbb{R}^{n+1}$ be a complete anisotropic
isoparametric hypersurface with $g$ distinct constant anisotropic
principal curvatures $\lambda_1,\cdots,\lambda_g$ of multiplicities
$m_1,\cdots,m_g$. Then we have
\begin{lem}\label{glessthan3}
$g\leq2$, and one anisotropic principal curvature must be $0$ if
$g=2$.
\end{lem}
\begin{proof}
As $A_F$ is a positive self-dual operator, its square root $C_F$ is
uniquely determined and is also a positive self-dual operator.
Without loss of generality, we can choose a local orthonormal basis
$e_1,\cdots,e_n$ of $x: M\rightarrow \mathbb{R}^{n+1}$ such that
under this basis,
$$C_FTC_F=diag(\lambda_1I_{m_1},\cdots,\lambda_gI_{m_g})=:\Lambda,$$
where $T$ is the shape operator (matrix) of $M$.

Let
$(\varepsilon_1,\cdots,\varepsilon_n):=(C_F(e_1),\cdots,C_F(e_n))=(e_1,\cdots,e_n)C_F$.
Then it is a tangent frame (not necessarily orthornormal) of $M$,
and
\begin{equation}\label{eigenvector}
(S_F(\varepsilon_1),\cdots,S_F(\varepsilon_n))=(e_1,\cdots,e_n)A_FTC_F=(\varepsilon_1,\cdots,\varepsilon_n)\Lambda,
\end{equation}
which shows that $\varepsilon_1,\cdots,\varepsilon_n$ are the
eigenvectors corresponding to the anisotropic principal curvatures,
\emph{i.e.}, they span the distributions $D_1,\cdots,D_g$
sequentially.

Now without loss of generality, suppose $g\geq2$ and $\lambda_1>0$
be the smallest positive anisotropic principal curvature. Then we
will show that $g=2$ and $\lambda_2=0$. For $t=1/\lambda_1$,
consider the degenerate anisotropic parallel translation
$x_t:M\rightarrow\mathbb{R}^{n+1}$ of $x$. As we showed in the proof
of Proposition \ref{prop}, $x_t(M)$ is an $(n-m_1)$-dimensional
submanifold immersed in $\mathbb{R}^{n+1}$ whose tangent space is
spanned by $D_2,\cdots,D_g$ up to translations\footnote{In general,
$D_1$ is not the normal space of $x_t(M)$.}. Thus we would like to
study the geometry of $x_t(M)$ by its second fundamental form. It
follows from (\ref{d parallel}) and (\ref{eigenvector}) that
\begin{equation}\label{1rd f f}
(\langle dx_t(\varepsilon_a),
dx_t(\varepsilon_b)\rangle)=(I-t\widetilde{\Lambda})\widetilde{A_F}(I-t\widetilde{\Lambda}),
\end{equation}
where the indices $a,b,\cdots\in\{m_1+1,\cdots,n\}$,
$\widetilde{\Lambda}:=diag(\lambda_2I_{m_2},\cdots,\lambda_gI_{m_g})$,
$\widetilde{A_F}:=(\langle\varepsilon_a,\varepsilon_b\rangle)=(A_{Fab})$
is the submatrix of $A_F$ with respect to
$Span(e_{m_1+1},\cdots,e_n)$ and thus is positive. Let
$\widetilde{C_F}$ be the square root of $\widetilde{A_F}$. Then it
can be easily verified from (\ref{1rd f f}) that
$$(\tilde{e}_{m_1+1},\cdots,\tilde{e}_n):=(\varepsilon_{m_1+1},\cdots,\varepsilon_n)(I-t\widetilde{\Lambda})^{-1}\widetilde{C_F}^{-1}$$
is an orthonormal basis of $x_t(M)$ under the induced metric. From
now on, we would like to use upper indices to denote elements of the
inverse of a matrix. For example, we write
$\widetilde{C_F}^{-1}=(\widetilde{C_F}^{ab})$. Then by (\ref{d
parallel}), we have
\begin{equation}\label{dxtea}
dx_t(\tilde{e}_a)=\sum_{c=m_1+1}^ndx_t(\varepsilon_c)(1-t\lambda_c)^{-1}\widetilde{C_F}^{ca}=\sum_{c=m_1+1}^n\varepsilon_c\widetilde{C_F}^{ca}.
\end{equation}
Recall that $\nu$ is a unit normal vector field of $x_t(M)$ and we
have
\begin{equation}\label{dvea}
-d\nu(\tilde{e}_b)=A_F^{-1}S_F(\tilde{e}_b)=\sum_{c=m_1+1}^nA_F^{-1}(\varepsilon_c)\lambda_c(1-t\lambda_c)^{-1}\widetilde{C_F}^{cb}.
\end{equation}
Taking inner product of (\ref{dxtea}) and (\ref{dvea}), we obtain
the second fundamental form $II_{\nu}$ of $x_t(M)$ in direction
$\nu$ as follows:
$$II_{\nu}(\tilde{e}_a,\tilde{e}_b)=\langle dx_t(\tilde{e}_a),-d\nu(\tilde{e}_b)\rangle=\widetilde{C_F}^{ca}\lambda_c(1-t\lambda_c)^{-1}\widetilde{C_F}^{cb},$$
or in matrix form,
\begin{equation}\label{2nd f f}
II_{\nu}=\widetilde{C_F}^{-1}\frac{\widetilde{\Lambda}}{I-t\widetilde{\Lambda}}~\widetilde{C_F}^{-1}.
\end{equation}
For any point $q\in x_t(M)$ and any unit normal vector $u\in
\mathbb{S}^{m_1}\subset\mathcal {V}_q(x_t(M))$, by Proposition
\ref{prop} and Lemma \ref{lemma}, there exists a unique point $p$ in
the leaf $L_1=x_t^{-1}(q)=W^{m_1}_F$ of the distribution $D_1$ such
that $\nu(p)=u$. Hence, for $u,-u\in \mathbb{S}^{m_1}\subset\mathcal
{V}_q(x_t(M))$, there exist $p_1,p_2\in L_1$ such that $\nu(p_1)=u$
and $\nu(p_2)=-u$. Therefore, both $II_u$ and $II_{-u}$ can be
expressed in the form (\ref{2nd f f}), although the matrix
$\widetilde{C_F}^{-1}$ may differ from each other and thus we denote
them simply by $\widetilde{C_{F1}}^{-1},\widetilde{C_{F2}}^{-1}$. On
the other hand, we have $tr(II_{-u})=-tr(II_u).$ Thus we derive the
following Cartan-type identity:
\begin{equation*}
\sum_{a=m_1+1}^n(\widetilde{A_{F1}}^{aa}+\widetilde{A_{F2}}^{aa})\frac{\lambda_a}{1-t\lambda_a}=0,
\end{equation*}
or equivalently,
\begin{equation}\label{cartan identity}
\sum_{k=2}^g\frac{\Gamma_F^k\lambda_k}{1-t\lambda_k}=0,
\end{equation}
where $\widetilde{A_{Fl}}^{aa}=(\widetilde{C_{Fl}}^{-2})^{aa}$s are
the diagonal elements of $\widetilde{A_F}^{-1}$ corresponding to the
point $p_l$, $l=1,2$,
$\Gamma_F^k:=\sum\limits_{a=n_{k-1}+1}^{n_k}(\widetilde{A_{F1}}^{aa}+\widetilde{A_{F2}}^{aa})>0$,
and $n_k=\sum\limits_{j=1}^km_j$, $k=1,\cdots,g$.

Recall that $\lambda_1$ is the smallest positive anisotropic
principal curvature and $t=1/\lambda_1$. Consequently each term in
the summation of (\ref{cartan identity}) is non-positive and thus
must vanish, which implies that $g=2$ and $\lambda_2=0$. The proof
is now completed.
\end{proof}
An immediate consequence of (\ref{2nd f f}) and Lemma
\ref{glessthan3} is the following.
\begin{cor}\label{totally geodesic xtM}
When $g=2$, $x_t(M)$ $(t=1/\lambda_1)$ is totally geodesic and thus
it is congruent to $\mathbb{R}^{n-m_1}\subset\mathbb{R}^{n+1}$.
\end{cor}
\begin{rem}
As the way taken in \cite{GT}, one can also calculate the power
expansion of $F$-Weingarten operator $S^t_F$ defined in
(\ref{F-shape for xt}) with respect to $t$ for $t$ sufficiently
close to $1/\lambda_1$, so as to deduce the Cartan-type identity
(\ref{cartan identity}) and hence Lemma \ref{glessthan3} and
Corollary \ref{totally geodesic xtM}.
\end{rem}
Finally, combining Lemma \ref{glessthan3}, Proposition \ref{totally
umbilical}, Corollary \ref{totally geodesic xtM} and formulas
(\ref{form WkF}), (\ref{subwulff-imm}), we can conclude the
classification in Theorem \ref{thm}. In view of results in Section
\ref{aniso-parallel-trans}, it seems that the local version of this
classification should also hold for analytic anisotropic defining
functions since in this case one may possibly derive the Cartan-type
identity (\ref{cartan identity}) by analytic extension.

\begin{ack}
The authors wish to thank Professor Haizhong Li for calling their attention to this problem
and Professor Reiko Miyaoka for her interest on this work.
They also would like to thank Professors Miyukai Koiso, Bennett Palmer and Zizhou Tang for
their supports and helpful conversations.

\end{ack}


\end{document}